\documentclass{article}%
\usepackage{amsfonts}
\usepackage{amsmath}
\usepackage{amsthm}
\usepackage{amssymb}
\usepackage{graphicx}%
\theoremstyle{plain}

\newtheorem{theorem}{Theorem}
\newtheorem*{thm}{Theorem}
\newtheorem*{proposition}{Proposition}
\newtheorem{lemma}{Lemma}
\theoremstyle{definition}

\newtheorem*{acknowledgement}{Acknowledgement}
\newtheorem{question}{Open question}
\theoremstyle{remark}

\newtheorem*{remark}{Remark}

\newcommand{\nc}{\newcommand}
\newcommand{\df}{\def}
\RequirePackage{ifthen}

\df\ie#1{{\mbox{{\it i.\hspace{0.06em}e.}\ifthenelse{\equal{#1}{,}}{\hspace{0.06em},}{\ #1}}}}
\df\eg#1{{\mbox{{\it e.\hspace{0.06em}g.}\ifthenelse{\equal{#1}{,}}{\hspace{0.06em},}{\ #1}}}}

\nc{\cit}[2][]{\ifthenelse{\equal{#1}{}}{{\rm \cite{#2}}}{{\rm \cite[#1]{#2}}}}

\makeatletter
\newlength{\JR@setLa}
\newlength{\JR@setLb}
\newlength{\JR@setLc}
\newlength{\JR@setLd}
\newlength{\JR@setLe}
\nc{\JR@setMta}{}
\nc{\JR@setMtb}{}
\nc{\JR@setMtc}{}
\nc{\JR@setMtd}{}

\nc{\JR@setMt}[1]{
\ifthenelse{#1=1}{\JR@setMta}{}
\ifthenelse{#1=2}{\JR@setMtb}{}
\ifthenelse{#1=3}{\JR@setMtc}{}
\ifthenelse{#1=4}{\JR@setMtd}{}
}

\nc{\JR@max}[3]{\ifthenelse{\lengthtest{\the#1>\the#2}}
           {\setlength{#3}{#1}}
           {\setlength{#3}{#2}}}

\nc{\JR@setds}{}
\nc{\JR@setSD}[4][]{
 \ifthenelse{\equal{#1}{}}
   {\renewcommand{\JR@setds}{}}
   {\renewcommand{\JR@setds}{\displaystyle}}
 \settoheight{\JR@setLa}{\ensuremath{\JR@setds{#2}}}
 \settoheight{\JR@setLb}{\ensuremath{\JR@setds{#3}}}
 \settodepth{\JR@setLc}{\ensuremath{\JR@setds{#2}}}
 \settodepth{\JR@setLd}{\ensuremath{\JR@setds{#3}}}
 \settowidth{\JR@setLe}{\ensuremath{\JR@setds{\left.\right.}}}
 \JR@max{\JR@setLc}{\JR@setLd}{\JR@setLc}
 \JR@max{\JR@setLa}{\JR@setLb}{\JR@setLa}
 \addtolength{\JR@setLa}{\JR@setLc}
 \ifthenelse{#4=1}{\renewcommand{\JR@setMta}{\rule[-\JR@setLc]{0pt}{\JR@setLa}}}{}
 \ifthenelse{#4=2}{\renewcommand{\JR@setMtb}{\rule[-\JR@setLc]{0pt}{\JR@setLa}}}{}
 \ifthenelse{#4=3}{\renewcommand{\JR@setMtc}{\rule[-\JR@setLc]{0pt}{\JR@setLa}}}{}
 \ifthenelse{#4=4}{\renewcommand{\JR@setMtd}{\rule[-\JR@setLc]{0pt}{\JR@setLa}}}{}
}

\df\SetEx[#1,#2](:#3|#4:){
 \JR@setSD[#2]{#3}{#4}{#1}
 \left\{\JR@setMt{#1} #3\right.\left|\rule{.9\JR@setLe}{0pt}#4 \JR@setMt{#1}\right\}
}

\nc{\set}[1][1]{\SetEx[#1,ds]}
\nc{\seti}[1][1]{\SetEx[#1,]}

\makeatother

\ifx\pdfoutput\relax\let\pdfoutput=\undefined\fi
\newcount\msipdfoutput
\ifx\pdfoutput\undefined\else
\ifcase\pdfoutput\else
\ifx\paperwidth\undefined\else
\ifdim\paperheight=0pt\relax\else\pdfpageheight\paperheight\fi
\ifdim\paperwidth=0pt\relax\else\pdfpagewidth\paperwidth\fi
\fi\fi\fi
\begin{document}

\author{Jo\"{e}l Rouyer}
\title{Steinhaus conditions for convex polyhedra}
\maketitle

\begin{abstract}
On a convex surface $S$, the antipodal map $F$ associates to a point $p$ the
set of farthest points from $p$, with respect to the intrinsic metric. $S$ is
called a Steinhaus surface if $F$ is a single-valued involution. We prove that
any convex polyhedron has an open and dense set of points $p$ admitting a
unique antipode $F_{p}$, which in turn admits a unique antipode $F_{F_{p}}$,
distinct from $p$. In particular, no convex polyhedron is Steinhaus.

\end{abstract}



\bigskip

\section{Introduction}

By a \emph{polyhedron}, denoted by $P$, we mean the boundary of a compact and
convex polyhedron in $\mathbb{R}^{3}$. $P$ is naturally endowed with its
\emph{intrinsic metric}: the distance between two points is the length of the
shortest curve joining them. In this paper, we shall never consider the
extrinsic distance. A \emph{segment} is by definition a shortest path on the
polyhedron between its endpoints. In general, it is not a line segment of
$\mathbb{R}^{3}$, but becomes so if one unfolds the faces it crosses onto a
same plane. An \emph{antipode} of $p$ is a farthest point from $p$; the set of
antipodes of $p$ is denoted by $F_{p}$. It is well-known that the mapping $F$
is upper semicontinuous. When the context makes clear that $F_{p}$ is a
singleton, we shall not distinguish between this singleton and its only element.

The study of antipodes on convex surfaces began with several questions of H.
Steinhaus, reported in \cite{CHF}, most of them answered by Tudor Zamfirescu,
see \eg {} \cite{Z1}, \cite{Z2}, \cite{Z4}, \cite{Z3}. However, one of those
questions had remained open a little longer: does the fact that the antipodal
map of a convex surface is a single-valued involution imply that the surface
is a round sphere? As we shall see, the answer is negative. By definition,
such a surface will be called a \emph{Steinhaus surface}.

The first family of Steinhaus surfaces was discovered by C. V\^{\i}lcu
\cite{V1}. It consists of centrally symmetric surfaces of revolution, and
includes the ellipsoids having two axes equal, and the third shorter than the
two equal ones. (Note that if the third axis is longer than the two equal
ones, the surface is no longer Steinhaus \cite{VZ1}.) Other examples were
discovered afterward: cylinders of small height \cite{IV3}, and the boundaries
of intersections of two solid balls, provided that the part of the surface of
the smaller ball included in the bigger one does not exceed a hemisphere. This
last example, as well as the mentioned family, were generalized to
hypersurfaces in $\mathbb{R}^{n}$ \cite{IRV1}.

Note that Steinhaus conditions are related to another one. Define the
\emph{radius} at a point $p$ as the distance between $p$ and its antipodes. It
is known that, if some surface has a constant radius map, then it is a
Steinhaus surface \cite{VZ1}. Moreover, all examples of Steinhaus surfaces
hitherto discovered also satisfy the constant radius condition. So it is still
open whether the two conditions are equivalent.

One can also notice that all known examples are surfaces of revolution, and in
particular, are not polyhedral. The first attempt to find a polyhedral example
was the investigation of the regular tetrahedron. An explicit computation of
the antipodal map proved that it is not a Steinhaus surface \cite{JR1}. Then
we proved that no tetrahedra can be Steinhaus \cite{JR4}. A few years later,
we proved that no polyhedron can satisfy the radius condition, and that no
centrally symmetric polyhedron can be Steinhaus \cite{JR5}. Since then, the
problem of the (non)existence of Steinhaus polyhedra was very natural. The aim
of this paper to prove they do not exist.

In order to make this article self-contained, we give in Section \ref{S2} some
preliminary, concerning the antipodes on a convex polyhedron. In Section
\ref{S3}, we prove the result. Then, in a last section, we discuss a few open
questions related to this topic.

\section{Preliminaries\label{S2}}

The explicit computation of the antipodal map in the case of the regular
tetrahedron \cite{JR1} was generalized -- as much as possible -- to arbitrary
convex polyhedra \cite{JR4},\cite{JR5}. We proved there the following theorem. \addtocounter{theorem}{-1}

\begin{theorem}
\label{T0}Any convex polyhedron $P$ can be written as a disjoint union%
\[
P=\Gamma\left(  P\right)  \cup\bigcup_{i=1}^{N}Z_{i}%
\]

with the following properties.

\begin{enumerate}
\item The sets $Z_{i}$ are open, and the restricted maps $F|Z_{i}$ are singled-valued.

\item The set $\Gamma\left(  P\right)  $ is a finite union of algebraic arcs,
of degree at most $10$. It includes the edges of $P$.

\item The map $F|Z_{i}$ is a constant map if and only if its image is a
singleton containing a vertex.

\item If $\operatorname{Im}\left(  F|Z_{i}\right)  =\left\{  v\right\}  $,
then there is exactly one segment between $v$ and any point of $Z_{i}$;
otherwise, there are exactly three segments between $p\in Z_{i}$ and $F_{p}$.
\end{enumerate}
\end{theorem}

From now on, the sets $Z_{i}$ defined in the above theorem will be referred to
as \emph{zones}.

\begin{remark}
Including a priori the edges in $\Gamma\left(  P\right)  $ is only a trick to
simplify the proof and ensure that any zone is (isometric to) a plane domain.
In general, nothing special happens while crossing an edge: if $Z$ and
$Z^{\prime}$ are two zones separated by an edge and one rotates the face of
$Z$ onto the plane of $Z^{\prime}$, then $F|Z$ and $F|Z^{\prime}$ become
rational prolongations of each other.

More generally, the notion of edges belongs to the extrinsic geometry and
plays no role in a purely intrinsic problem.
\end{remark}

The proof of Theorem \ref{T0} is given in details in \cite{JR4} and
\cite{JR5}, we present next only its rough idea. The fourth point of the
theorem has been stated for aim of completeness only. It will not be used
here, so we shall omit its proof which involves extensive computation.

A first remark is that a segment cannot pass through any vertex, see for
instance \cite{BCS}.

Secondly, note that a point $p$ on the polyhedron is joined to any of its
antipodes $q$ by at least three segments, provided that $q$ is not a vertex.
Suppose on the contrary that there are only two segments $\sigma_{1}$ and
$\sigma_{2}$ between $p$ and $q$. They would determine two sectors at point
$q$, one of them having measure at least $\pi$ (if there is only one segment,
there is only one sector, measuring $2\pi$). Consider a point $r$ on the
bisector of this sector, tending to $q$. A segment $\sigma$ between $p$ and
$r$ should tend to either $\sigma_{1}$ or $\sigma_{2}$, say $\sigma_{1}$. For
$r$ close enough to $q$, the triangle composed of $\sigma_{1}$, $\sigma$, and
the only segment between $r$ and $q$ contains no vertex, and so, is a (folded)
Euclidean triangle. Moreover the angle at $q$ is obtuse or right; it follows
that the length of $\sigma$ (which is also $d\left(  p,r\right)  $) is greater
than the length of $\sigma_{1}$ (which is supposed to be the radius at $p$),
and we get a contradiction.

Hence there exist three segments $\sigma_{-1}$, $\sigma_{0}$, $\sigma_{1}$
between $p$ and $q$. Let $\mathcal{F}^{\prime}$ be the face of $p$ and
$\mathcal{F}$ be the face of $q$; if one unfolds the union of the faces
crossed by $\sigma_{i}$ ($i=-1$, $0$, $1$) onto the plane of $\mathcal{F}$,
one obtains three images of the face $\mathcal{F}^{\prime}$, say
$\mathcal{F}_{0}$, $\mathcal{F}_{-1}$ and $\mathcal{F}_{1}$. One passes from
$\mathcal{F}_{0}$ to $\mathcal{F}_{i}$ ($i=\pm1$) by a planar affine direct
isometry $f_{i}$. Since the segments $\sigma_{i}$ have the same length,
$q=\mathrm{cc}\left(  p_{0},f_{-1}\left(  p_{0}\right)  ,f_{1}\left(
p_{0}\right)  \right)  $, where $p_{0}$ is the point of $\mathcal{F}_{0}$
corresponding to $p$ and $\mathrm{cc}\left(  ~,~,~\right)  $ stands for the
circumcenter (see Figure \ref{F1}).

There is only a finite number of ways to unfold sequences of faces, inducing a
finite number of pairs of isometries $f_{\pm1}$, leading to a finite number of
maps $\tau_{k}:$ $p\mapsto\mathrm{cc}\left(  p_{0},f_{-1}\left(  p_{0}\right)
,f_{1}\left(  p_{0}\right)  \right)  $. For each zone $Z$ the function $F|Z$
is either one $\tau_{k}$ or a constant map whose value is a vertex. Let
$\delta_{k}\left(  p\right)  $ be the square of the distance between $p$ and
$\tau_{k}\left(  p\right)  $. The algebraic arcs that compose $\Gamma\left(
P\right)  $ are parts of the locus of those points such that $\delta
_{k}\left(  p\right)  =\delta_{k^{\prime}}\left(  p\right)  $ for some indices
$k$ and $k^{\prime}$, such that $\delta_{k}\left(  p\right)  =d\left(
p,v\right)  ^{2}$ for some index $k$ and some vertex $v$, or such that
$d\left(  p,v\right)  ^{2}=d\left(  p,w\right)  ^{2}$ for some vertices $v$
and $w$. A\ straightforward computation shows that they have degree at most
$10$. (Note that, in the only case that have been explicitly solved -- the
regular tetrahedron -- the higher degree terms vanish and the actual degree
drops to $4$ \cite{JR1}. So it is not entirely clear that $10$ can be achieved.)

\bigskip\begin{figure}[tb]
\begin{center}
\includegraphics[
scale=.4
]{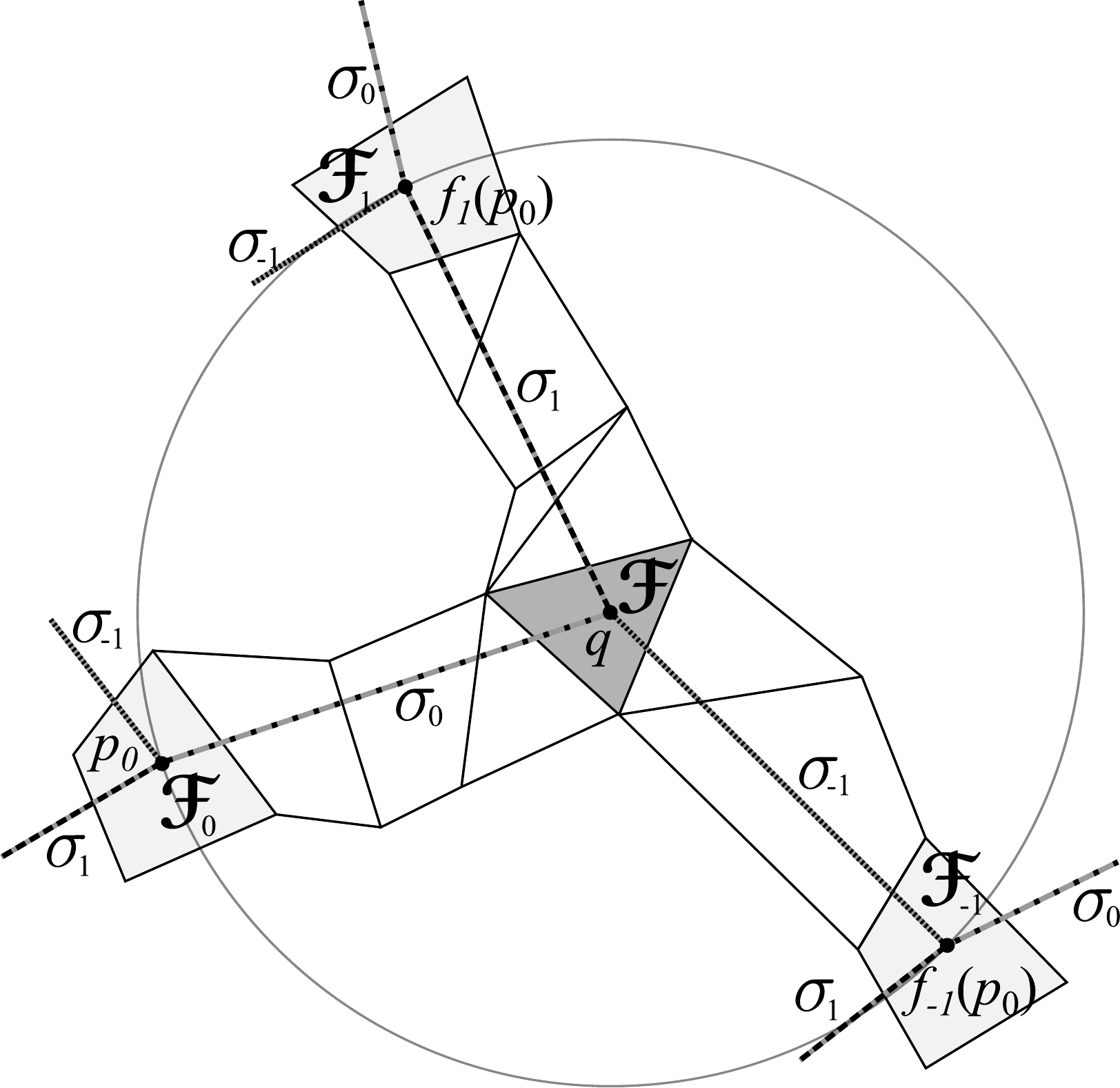}
\end{center}
\caption{Unfolding of $P$ in the proof of Theorem \ref{T0}.}%
\label{F1}%
\end{figure}

For further reasoning, it is necessary to notice that $f_{1}$ and $f_{-1}$
cannot be both translations (Lemma 1 in \cite{JR5}). Moreover, if either
$f_{1}$ or $f_{-1}$ is a translation then, by interchanging the roles of
$F_{0}$ and $F_{1}$, we obtain that $f_{-1}$ and $f_{1}$ are both rotations,
of the same angle. Hence, we can assume without loss of generality that
$f_{\pm1}$ are both rotations. Now, a straightforward computation proves the
following lemma.

\begin{lemma}
\label{L1} \cit{JR4} Let $P$ be a convex polyhedron, let $Z$ be a zone of $P$,
and assume that $F|Z$ is not constant. Endow the plane of $Z$ and the plane of
$\operatorname{Im}\left(  F|Z\right)  $ with orthonormal frames and let
$\left(  x,y\right)  $ be the coordinates of $p\in Z$. The coordinates of
$F_{p}$ are given by the formula:%
\begin{equation}
F\left(  x,y\right)  =\frac{\left(  X\left(  x,y\right)  ,Y\left(  x,y\right)
\right)  }{\varepsilon\left(  x^{2}+y^{2}\right)  +L\left(  x,y\right)
}\text{,}\label{3}%
\end{equation}
where $\varepsilon\in\left\{  0,1\right\}  $, $L$, $X$, $Y\in\mathbb{R}\left[
x,y\right]  $, $\deg\left(  X\right)  \leq2$, $\deg\left(  Y\right)  \leq2$,
$\deg\left(  L\right)  \leq1$. Moreover, the zero set of the denominator is
neither empty nor reduced to a single point.
\end{lemma}

Indeed, $\varepsilon=0$ if and only if the function $f_{\pm1}$ are rotations
of the same angle. In this case, the zero set of the denominator of (\ref{3})
is the line through the centers of $f_{1}$ and $f_{-1}$. If the angles of the
rotations are distinct (\ie, $\varepsilon=1$), the zero set of the denominator
is the circle through the centers of $f_{1}$, $f_{-1}$, and $f_{-1}\circ
f_{1}^{-1}$.

\begin{lemma}
\label{L2}With the notation of Lemma \ref{L1}, in the case $\varepsilon=0$,
there exists orthonormal frames of the planes of $Z$ and $F_{Z}$ such that
\begin{align}
X\left(  x,y\right)   &  =2xy\cos\theta+\left(  x^{2}-y^{2}-1\right)
\sin\theta\nonumber\\
Y\left(  x,y\right)   &  =(1-\cos\theta)\left(  x+y\cot\frac{\theta}%
{2}-1\right)  \left(  x+y\cot\frac{\theta}{2}+1\right) \label{2}\\
L\left(  x,y\right)   &  =2y\text{.}\nonumber
\end{align}

\end{lemma}

\begin{proof}
Let $\theta$ be the angle common to both rotations. Choose the orthonormal
direct frame in the plane of $\operatorname{Im}\left(  F|Z\right)  $ in such a
way that the coordinates of the center of $f_{i}$ are $\left(  i,0\right)  $
($i=\pm1$). In the plane of $Z$, choose the frame in which the coordinates of
$p\in\mathcal{F}$ equal the coordinates of $p_{0}\in\mathcal{F}_{0}$. Let
$\left(  x,y\right)  $ be the coordinates of $p$; the coordinates of
$f_{i}\left(  p_{0}\right)  $ are
\[
\left(  \left(  x-i\right)  \cos\theta-y\sin\theta+i,y\cos\theta\left(
x-i\right)  \sin\theta\right)  \text{.}%
\]
The equation in coordinates $\left(  \xi,\psi\right)  $ of the mediator of the
segment $p_{0}f_{i}\left(  p_{0}\right)  $ is%
\begin{align}
&  \left(  x\left(  \cos\theta-1\right)  -y\sin\theta-i\left(  \cos
\theta-1\right)  \right)  \xi\nonumber\\
+ &  \left(  y\left(  \cos\theta-1\right)  +(x-i)\sin\theta\right)
\psi\nonumber\\
+ &  iy\sin\theta+\left(  \cos\theta-1\right)  \left(  1-ix\right)
=0\text{.}\label{4}%
\end{align}
In order to shorten the formulas, we put $u\overset{\mathrm{def}}{=}\cos
\theta-1=-2\sin^{2}\frac{\theta}{2}$ and $v\overset{\mathrm{def}}{=}\sin
\theta=2\sin\frac{\theta}{2}\cos\frac{\theta}{2}$. The half difference and
half sum of the two equations given by (\ref{4}) ($i=\pm1$) are respectively%
\[%
\begin{array}
[c]{rrr}%
u\xi & +v\psi & =-ux+vy\\
\left(  ux-vy\right)  \xi & +\left(  uy+vx\right)  \psi & =-u\text{.\hspace
{0.9cm}}%
\end{array}
\]
It follows that
\begin{align*}
\xi &  =\frac{\left(  -ux+vy\right)  \left(  uy+vx\right)  +uv}{\left(
uy+vx\right)  u-\left(  ux-vy\right)  v}\\
&  =\frac{\left(  v^{2}-u^{2}\right)  xy+uv\left(  y^{2}-x^{2}+1\right)
}{y\left(  u^{2}+v^{2}\right)  }=\frac{X\left(  x,y\right)  }{2y}%
\end{align*}%
\begin{align*}
\psi &  =\frac{\left(  -ux+vy\right)  \left(  ux-vy\right)  +u^{2}}{v\left(
ux-vy\right)  -u\left(  uy+vx\right)  }\\
&  =\frac{-u^{2}x^{2}+2uvxy-v^{2}y^{2}+u^{2}}{2yu}\\
&  =\frac{u\left(  1-x+\frac{v}{u}y\right)  \left(  1+x-\frac{v}{u}y\right)
}{2y}=\frac{Y\left(  x,y\right)  }{2y}\text{.}%
\end{align*}

\end{proof}

\section{Main result\label{S3}}

Now, we are in a position to prove the claimed result. The idea of proof is
very simple: we just have to prove that the inverse of a function of the form
(\ref{3}) cannot be of this form.

\begin{thm}
Any convex polyhedron $P$ contains an open and dense set $D$ such that $F|D$
is single-valued, $F|F_{D}$ is single-valued, and $F_{F_{x}}\neq x$ for all
$x\in D$. Consequently, no polyhedron satisfies, even locally, Steinhaus conditions.
\end{thm}

\begin{proof}
It is sufficient to prove that any open set $U_{0}\subset P$ contains an open
subset $V$ such that $F|V$ and $F|F_{V}$ are single-valued, and such that
$F_{F_{x}}\neq x$ for any $x\in V$. Then, the union of all those $V$ is a
suitable $D$.

Let $U_{0}$ be a fixed open set; by virtue of Theorem \ref{T0}, it contains a
subset $U_{1}$ which is included in a zone. Since $F|U_{1}$ is continuous,
there exists a smaller open set $U\subset U_{1}$ such that $F_{U}$ is wholly
contained in a zone too. Since $F\circ F|U$ is single-valued and thereby
continuous, the set $V$ of those points in $x\in U$ such that $F_{F_{x}}\neq
x$ is open. If it is non-empty, the proof is over; suppose on the contrary
that $F_{F_{p}}=\left\{  p\right\}  $ for all $p\in U$. Put $U^{\prime
}\overset{\mathrm{def}}{=}F_{U}$. Let $\Pi$ be the plane of the face
containing $U$ and $\Pi^{\prime}$ be the plane of the face containing
$U^{\prime}$. By Lemma \ref{L1}, $F|U$ and $F|U^{\prime}$ are rational and
admit natural continuations $G:\Pi\backslash C\rightarrow\Pi^{\prime}$ and
$G^{\prime}:\Pi^{\prime}\backslash C^{\prime}\rightarrow\Pi$ respectively,
where $C$ and $C^{\prime}$ are either a circle or a line. By hypothesis,
$G^{\prime}|U^{\prime}\circ G|U=id_{U}$; since $G$ and $G^{\prime}$ are
rational, $G^{\prime}\circ G\left(  p\right)  =p$ wherever $G^{\prime}\circ G$
is well defined.

Assume that $\operatorname{Im}G$ intersects $C^{\prime}$. Let $a$ be a point
of the boundary of $G^{-1}\left(  C^{\prime}\right)  $, and $p_{n}\in
\Pi\backslash\left(  C\cup G^{-1}\left(  C^{\prime}\right)  \right)  $ be a
point tending to $a$ when $n$ tends to infinity. Then $\left\Vert G^{\prime
}\circ G\left(  p_{n}\right)  \right\Vert $ tends to infinity, in
contradiction with $G^{\prime}\circ G\left(  p_{n}\right)  =p_{n}\rightarrow
a$. Hence $\operatorname{Im}G$ does not intersect $C^{\prime}$, and similarly
$\operatorname{Im}G^{\prime}$ does not intersect $C$.

Assume now that either $C$ or $C^{\prime}$ is a circle. Since $G$ and
$G^{\prime}$ play symmetrical roles, we can assume that $C$ is a circle. Let
$D$ be the interior of $C$ and $E$ its exterior; let $A$, $B$ be the two
connected components of $\Pi^{\prime}\backslash C^{\prime}$. Since $\left\Vert
G^{\prime}\left(  x\right)  \right\Vert $ tends to infinity when $x$ tends to
some points of $C^{\prime}$, $G^{\prime}\left(  A\right)  $ (resp. $G^{\prime
}\left(  B\right)  $) cannot be included in $D$. It follows that
$\operatorname{Im}G^{\prime}\subset E$. This contradict the fact that a point
$x\in D$ equals $G^{\prime}\circ G\left(  x\right)  $.

Now assume that $C$ and $C^{\prime}$ are both straight lines. By Lemma
\ref{L2}, one can assume that the expression of $G$ in Euclidean coordinates
is given by (\ref{3}) and (\ref{2}). Let $p_{1}\in\Pi$ be the point of
coordinates $\left(  \cos\theta,-\sin\theta\right)  $, $p_{2}\in\Pi$ be the
point of coordinates $\left(  -\cos\theta-2,\sin\theta\right)  $ and $q\in
\Pi^{\prime}$ the point of coordinates $\left(  1,0\right)  $. A direct
computation shows that%
\[
G\left(  p_{1}\right)  =G\left(  p_{2}\right)  =q\text{.}%
\]
Hence $p_{1}=G^{\prime}\left(  q\right)  =p_{2}$ and we get a contradiction.
\end{proof}

\section{Further questions}

In order to close the paper, we will mention a few open questions related to
the above result. The first two were already suggested in the introduction.

\begin{question}
Do there exist Steinhaus surfaces with non-constant radius?
\end{question}

\begin{question}
Do there exist Steinhaus surfaces without rotational symmetry?
\end{question}

\smallskip

Observe that the definition of Steinhaus surfaces combines two distinct
conditions: the fact that $F$ is single-valued does not imply that it is an
involution. The simplest known example is an ellipsoid of revolution whose
axes measure respectively $1$, $1$, and $a\in\left(  1,\sqrt{2}\right)  $. On
such an ellipsoid, $F$ is a homeomorphism, but not an involution \cite{VZ1}.
This naturally leads to the following question.

\begin{question}
Does there exist a convex polyhedron such that $F$ is single-valued? Such that
$F$ is a homeomorphism?
\end{question}

\smallskip

A certain trend nowadays is to consider the so-called \emph{Alexandrov
surfaces with curvature bounded below}, instead of the classical convex
surfaces. Roughly speaking, the difference between these notions is twofold.
On the one hand, the curvature bound is no longer necessarily zero, and on the
other hand, the topology is no longer spherical. See \cite{BGP} for details.
In this context, it is natural to consider\emph{ abstract polyhedra} obtained
by gluing several polygons along their boundaries, in such a way that (1) the
gluing map is length preserving, (2) the resulting space is a (not necessarily
spherical) topological surface, and (3) the singular curvature at each vertex
is non-negative. It is easy to see that the theorem of this paper cannot be
generalized to such abstract polyhedra, for rectangle flat tori are Steinhaus.
As for projective planes, the main obstruction to be Steinhaus is purely topological.

\begin{proposition}
Any metric space homeomorphic to an even dimensional projective space admits
at least one point with more than one antipodes.
\end{proposition}

\begin{proof}
Assume that $F$ is single-valued, and consequently continuous. Since, as a
consequence of Lefschetz fixed point theorem, even dimensional projective
spaces have the fixed point property, $F$ should have a fixed point, which is absurd.
\end{proof}

Still in the context of Alexandrov surfaces, it is natural to consider
abstract polyhedra whose faces are geodesic polygons of the unit sphere
(\emph{spherical polyhedra}), or of the hyperbolic plane (\emph{hyperbolic
polyhedra}).

\begin{question}
Except from the sphere, is there a Steinhaus spherical polyhedron? A
hyperbolic one? If yes, which are the admissible topologies?
\end{question}

\bigskip

\begin{acknowledgement}
The author was supported by the grant PN-II-ID-PCE-2011-3-0533 of the Romanian
National Authority for Scientific Research, CNCS-UEFISCDI. \smallskip

Special thanks are due to the anonymous referee for his or her useful
suggestions that strongly influenced the final form of this paper.
\end{acknowledgement}

\providecommand{\bysame}{\leavevmode\hbox to3em{\hrulefill}\thinspace}
\providecommand{\MR}{\relax\ifhmode\unskip\space\fi MR }
\providecommand{\MRhref}[2]{%
  \href{http://www.ams.org/mathscinet-getitem?mr=#1}{#2}
}
\providecommand{\href}[2]{#2}

\bigskip

\noindent Jo\"{e}l Rouyer

\noindent{\small Simion Stoilow Institute of Mathematics of the Romanian
Academy, \newline P.O. Box 1-764, Bucharest 70700, ROMANIA \newline
Joel.Rouyer@ymail.com}

\end{document}